\documentclass[a4paper,12pt]{amsart}
\usepackage{amsmath,amsthm,amssymb,hyperref,graphicx,mathrsfs,mathtools}  
\theoremstyle{plain}
\newtheorem{theorem}{Theorem}
\newtheorem{corollary}[theorem]{Corollary}

\newtheorem{prop}[theorem]{Proposition}
\newtheorem{definition}[theorem]{Definition}
\newtheorem{example}[theorem]{Example}

\newtheorem{remark}[theorem]{Remark}

\allowdisplaybreaks
\numberwithin{equation}{section}
\title[Banach limits]{Banach limits - Some new thoughts and perspectives}
\author{M.A. Sofi}
\address{Department of Mathematics, University of Kashmir, Srinagar-190006, India}
\email{aminsofi@gmail.com}
\begin{document}
\maketitle
\begin{abstract}
The existence of a Banach limit as a translation invariant positive continuous linear functional on the space of bounded scalar sequences which is equal to 1 at the constant sequence $(1,1,\ldots,1,\ldots)$ is proved in a first course on functional analysis as a consequence of the Hahn Banach extension theorem. Whereas its use as an important tool in classical summability theory together with its application in the existence of certain invariant measures on compact (metric) spaces is well known, a renewed interest in the theory of Banach limits has led to certain applications which have opened new vistas in the structure of Banach spaces.

The paper is devoted to a discussion of certain developments, both classical and recent, surrounding the theory of Banach limits including the structure of the set of Banach limits with special emphasis on certain aspects of their applications to the existence of certain invariant measures, vector valued analogues of Banach limits, functional equations and in the structure theory of Banach spaces involving the existence of selectors of certain multi-valued mappings into the metric space of non-empty, convex, closed and bounded subsets of a Banach space with respect to the Hausdorff metric. The paper shall conclude with a brief description of some recent results of the author on the study of 'simultaneous continuous linear' operators (linear selections) involving Hahn Banach extensions on spaces of Lipschitz functions on (subspaces of) Banach spaces.  
\end{abstract}
\section{Introduction}
Some open problems that naturally arise in the study shall also be included. 
Some open problems that naturally arise in the study shall also be included. 
\begin{definition}
A Banach limit (also called invariant mean) is a positive shift invariant 
$\big(L(x_1, x_2\ldots, x_n,\ldots ,  ) =  (x_2, x_3, \ldots,x_{n+1},\ldots)\big)$ continuous linear functional $L$ on the space $\ell_\infty$ (of bounded sequences) such that $L(e) = 1$ $ (e = (1,1,1,\ldots,1,\ldots).$  
\end{definition}
\subsection{Existence of Banach limits}
One of the common methods to prove the existence of Banach limits is via the use of the invariant version of the Hahn Banach theorem whereas the first proof by M.M. Day had made use of Markov-Kakutani fixed-point theorem. However, an application of the extension form of Hahn Banach theorem applied to the limit functional on the subspace $c$ (of convergent sequences) of $\ell_\infty$ with respect to a certain sublinear functional $\Lambda$ on $\ell_\infty$ yields the existence of a Banach limit. Indeed, taking $\Lambda$ to be the sublinear functional defined by:
\begin{align*}
\Lambda(x)=\underset{n\to \infty}{\lim}a_n(x), \qquad x\in\ell_\infty
\end{align*}
where $a_n(x)=\underset{i\geq 1}{\sup}\sum_{j=0}^{n-1}x_{i+j},$ we note that if $n=km+r$ with $0\leq r<m,$ then for each $n>m,$ $a_n(x)\leq a_m(x)+\frac{m}{n}.$ It follows that the above limit exists and that 
\begin{align*}
\Lambda(x)=L(x)=\underset{n\to \infty}{\lim} x_n, \qquad\text{for} x~~\in c.
\end{align*}
Applying the Hahn Banach extension theorem to the limit functional $L$ on $c$ with respect to the sublinear functional $\Lambda$ gives a linear functional $F$ on $\ell_\infty$ dominated by $\Lambda$ and agreeing with $L$ on $c.$ Further, we note that $F$ is shift invariant. Indeed, let $x=(x_1, x_2,…,x_n,…)\in\ell_\infty$, and put
$y=(x_1, x_2,\ldots,x_n,\ldots)-(x_2,x_2,\ldots,x_n,\ldots).$ It follows that $\Lambda(y)=\Lambda(-y)=0$ and, therefore, $F(y) = 0.$ In other words, $F(x_1, x_2,\ldots,x_n,\ldots) = F(x_2,x_2,\ldots,x_n,\ldots).$ Finally, the fact that $F$ is dominated by $\Lambda$ gives that $F$ is positive and that $F(e) = 1.$ 

\subsection{Existence of Banach limits with further invariance properties: }
Besides the shift operator, there are classes of bounded linear operators $\Gamma$ on $\ell_\infty$  and Banach limits which are invariant with respect to $\Gamma$ in the sense that $BH=B$ for all $B\in\mathcal{B}$ and $H\in\Gamma.$ A useful choice for Γ is the subset of all operators $T\in L(\ell_\infty  )$ satisfying:
\begin{enumerate}
\item[(i)] $T$ is positive.
\item[(ii)] $T(e)\geq 0.$
\item[(iii)] $\underset{j\to \infty}{\limsup}\big(A(I-T)x\big)_j\geq 0,$ $x\in\ell_\infty,$ $A\in H=\text{conv}\{T^n,n=1,2,\ldots\}$ (See \cite{12}).
\end{enumerate}
An old result of Eberlein (\cite{4}) ensuring the existence of Banach limits which are invariant with respect to the Ces\`{a}ro operator $C$ follows as a special consequence where 
\begin{align*}
\big(C(x)\big)_n=\frac{1}{n}\sum_{i=1}^{n}x_i, \quad x\in\ell_\infty.
\end{align*}
\subsection{Extreme points of the set $\mathcal{B}$}
We shall denote by $\mathcal{B}$ the set of all Banach limits and by $\text{Ext}(\mathcal{B}),$ the set of its extreme points. We have the following theorems of Semenov-Sukochev (\cite{13}).

\begin{theorem}\label{sst1}
Each sequence of distinct points of  $\textnormal{Ext}(\mathcal{B})$ is 1-equivalent to the unit vector basis of $\ell_1.$
\end{theorem}

\begin{theorem}\label{sst2}
The closed convex hull of  $\textnormal{Ext}(\mathcal{B})$ does not contain a Ces\`{a}ro invariant Banach limit.
\end{theorem}

\subsection{Applications}
\noindent \textit{(a). Existence of certain invariant measures:}

\noindent (i) \textbf{Theorem:} 
Each compact metric space $X$ admits a regular Borel measure which is invariant under a given continuous mapping on $X.$ More generally, every compact metric space admits an isometry-invariant Borel probability measure. 

To see how the proof works for the first part of the statement, let $T:X\to X$ be a given continuous mapping and let $T^n$ denote the $n$th iterate of $T.$ Choose a Banach limit $\Lambda$ on $\ell_\infty$ and fix some $x\in X.$ Further, consider the linear functional $F$ on $C(X)$ given by
\begin{align*}
F(f)=\Lambda\left(f(x),f(T(x)),f(T^n(x)),\ldots\right).
\end{align*}
An application of the Reisz Representation Theorem yields a unique regular Borel measure $\mu$ on $X$ such that $F(f)=\int_{X} fd\mu$ for each $f\in C(X).$ Using the (shift) invariance of $\Lambda$ gives 
\begin{align*}
\Lambda\left(f(x),f(T(x)),f(T^3(x)),\ldots\right)=\Lambda\left(f(T(x)),f(T^2(x)),f(T^3(x)),\ldots\right),
\end{align*}
or equivalently,$\int_{X} fd\mu=\int_{X} f\circ Td\mu$ for each$f \in C(X).$ Finally, invoking the change of formula from calculus, we see that $\int_{X} f d\mu=\int_{X} f d\mu T^{-1}$  holds for each $f\in C(X).$ Finally, the uniqueness of the measure $\mu$ gives that $\mu=\mu T^{-1},$ yielding thereby that $\mu$ is T-invariant.

The latter part of the theorem follows from the existence of the Haar measure on a compact topological group. Indeed, noting that the group of isometries $\text{Iso}(X)$ on $X$ is compact under the sup-norm, there exists a Haar measure, say $\nu$ on $\text{Iso}(X).$ Fix $x\in X.$ Then the function $\mu$ on $P(X)$ defined by: $\mu(A)=\nu (\{g\in \text{Iso}(X):g(x)\in A\})$ is the desired measure.

\noindent (ii) There exists a finitely additive non-negative probability isometry-invariant measure on the $\sigma$- algebra of all subsets of $\mathbb{R}$ (resp.$\mathbb{R}^2$). (Such a measure is called a \textit{Banach measure}).

The above important result is proved by refining and extending the idea of a Banach limit where now a (semi)group $S$ is used in place of $\mathbb{R}$ thus necessitating using the Banach space $\ell_\infty (S)$ in place of $\ell_\infty$ and equipped with the supremum norm. More precisely, we have the following:
\begin{definition}
A (semi)group $S$ is said to be \textbf{\textit{(right) amenable}} if there exists a (right) \textbf{\textit{invariant mean}} on $\ell_\infty (S),$ i.e., a (norm-one) linear functional $F$ on $\ell_\infty (S)$ such that $F(e) =1$ and $F(f) = F(f^s)$ where $f^s (t)=f(ts),~t,s\in S.$ Similarly for the left invariant mean. Further, S is said to be \textbf{\textit{amenable}} if it is both right and left amenable.
\end{definition}
Thus the existence of a Banach limit on $\ell_\infty$ as shown in (1.1) above amounts to saying that $\mathbb{R}$ as a group is amenable. In fact, each abelian group, and more generally each solvable group is amenable. Also there are non-abelian amenable groups which include the class of solvable groups. A concrete example of such a group is provided by the rotation group $G =\text{ Iso}(\mathbb{R})$ of $\mathbb{R}$ (or $\text{Iso}(\mathbb{R}^2).$ Thus, there exists an invariant mean, say $F$ on $G$ and it is easy to check that the formula $\nu (A)=F(\chi_A)$ defines a Banach measure on $G.$  Finally, under the action of $G$ on $X = \mathbb{R}$ (resp. on $\mathbb{R}^2$), the desired $G$-invariant measure $\mu$ on $X$ is defined, as usual, by fixing $x\in X$ and taking $\mu (A)=\nu({g\in\text{Iso}(X):g(x)\in A}).$ This completes the argument involving the proof of (1.4(ii)) above. However, the result fails for higher dimensions as the isometry group $\text{Iso}(\mathbb{R}^n )$ is not amenable for $n\geq 3$ (see \cite[Chapter 18]{3}).  

\noindent \textit{(b). Classical Summability Theory}

Given a Banach limit $L,$ it follows that $L(x) = \underset{n\to \infty}{\lim}x_n,$ for each $ x\in c.$ 
There are (bounded) non-convergent sequences having this property, e.g. it can be checked that for
$x = (1, 0, 1, 0, \ldots),$ $L(x) = \frac{1}{2}$ for all Banach limits $L.$ A sequence with this property shall be designated as being almost convergent. This definition was proposed by G. Lorentz in 1948: 
\begin{definition}[\cite{10A}]
A sequence $x=(x_n)\in\ell_\infty$ is said to be almost convergent to $s\in\mathbb{R}$  (called the $F$-limit) if $L(x) = s$ for all Banach limits $L.$
\end{definition}
\begin{theorem}[\cite{10A}]
Given $x=(x_n)\in\ell_\infty$, then $F$-limit $x_n=s$ if and only if
\begin{align*}
\underset{m\to \infty}{\lim}\frac{1}{m}\left(\sum_{j=0}^{m-1}x_{n+j}\right)=s,
\end{align*}
uniformly in $n.$
 \end{theorem}
\begin{remark}
In the general case, Sucheston \cite{18} shows that for each $x=(x_n)\in\ell_\infty,$ we have:
$\{Lx;L\in\mathcal{B}\}=[p,q],$ where 
\begin{align*}
p=\underset{m\to \infty}{\lim}\underset{n\geq 1}{\inf}\dfrac{1}{m}\left(\sum_{j=0}^{m-1}x_{n+j}\right),~~q=\underset{m\to \infty}{\lim}\underset{n\geq 1}{\sup}\dfrac{1}{m}\left(\sum_{j=0}^{m-1}x_{n+j}\right)
\end{align*}
\end{remark}
\noindent \textit{(c). Stability of Cauchy functional equations}
\begin{definition}
Given a (semi)group $G,$ a mapping  $f:G\to G$ is said to be \textit{additive} if it satisfies the Cauchy functional equation:
\begin{align*}
f(x+y)=f(x)+f(y), \quad\text{for all}~~x,y\in G.
\end{align*}
\end{definition}
Hyers-Ulam stability theorem \cite{7} (see\cite[Chapter 15]{3}) states that for Banach spaces $X$ and $Y$ and a mapping $f:X\to Y$  such that
\begin{align}\label{star}
\|f(x+y)-f(x)-f(y)\|<\epsilon, \quad\text{for all}~~ x,y\in X, 
\end{align}
there exists a unique additive mapping $g:X\to  Y$ satisfying
\begin{align}\label{star2}
\|f(x)-g(x)\|<\epsilon, \forall x\in X.
\end{align}
\textbf{Question 1.} Given a semigroup $S$ and a Banach space $X,$ does it follow that for each $f:S\to X$ and $\epsilon>0$ satisfying \eqref{star}, there exists $g:S\to X$ for which \eqref{star2} holds for all $x,y\in S.$

When such is the case, we say that the Cauchy functional equation is stable. 

Question 1 has an affirmative answer for each (left) amenable semigroup $S$ and Banach spaces $X$ in each of the following cases where $X$ is:
\begin{enumerate}
\item[(i)] Reflexive (\cite{1}).
\item[(ii)] Injective (\cite{2}).
\end{enumerate}
The proof in each of the above cases depends upon the notion and the existence of vector-valued Banach limits on certain special spaces of $X$-valued bounded functions on a semigroup $S$ and a certain 'selection procedure'. (See Section (d) below).

\noindent \textit{(d) Vector valued Banach limits}

\noindent \textbf{Question 2.} Given a semigroup $S$ and a Banach space $X,$ is it true that there exists an $X$-\textit{valued (left)-(resp. right) invariant mean} on $S,$ i.e., a bounded linear map $M:B(S,X)\to X$ such that (i) $M(x\chi_S)=x$ for all $x\in X$ and (ii) $M(f) = M(f_S)$ (resp. $M(f)=M(f^S))$ where $B(S,X)$  denotes the Banach space of all $X$-valued bounded functions on $S$ and $x\chi_S$ stands for the function on $S$ which is identically equal to $x.$ By an invariant mean we shall mean a left-invariant mean which is also right invariant. 

The above definition of vector-valued invariant mean has been defined and investigated in the literature by several authors. Some recent developments along these lines and appropriate to the theme of the present work appear in \cite{9}. The case $S = \mathbb{N}$ and $ X=\mathbb{R}$ leads to the classical notion of the invariant mean. As we shall see below, vector-valued invariant mean always exists for a large class of Banach spaces. We will see that for the purpose of proving stability results involving Q1, it suffices to show the existence of an $X$-valued left (resp. right) invariant mean when we restrict ourselves to certain special subspaces of $B(S,X).$ The next result provides a sample of such results. Here $\overline{\text{co}}(A)$ denotes the weakly closed convex hull of $A.$

\begin{theorem}\textnormal{\cite{1}}\label{thm9}
(a) Let $X$ be a Banach space and $S$ a (left) amenable semigroup. Then the space $B_{wc} (S,X)$ of all $X$-valued functions on $S$ whose range has a weakly compact convex hull admits a (left) invariant mean $M:B_{wc} (S,X)\to X$ such that $M(f)\in\overline{\textnormal{co}}(f(S))$ for each $f\in B_{wc} (S,X).$  In particular, for reflexive Banach space $X$ and (left) amenable semigroup $S,$ there exists a (left) invariant mean $M:B(S,X)\to X$ with $M(f)\in\overline{\textnormal{co}}(f(S))$ for each $f\in B_{wc}(S,X).$
\end{theorem}

In an effort to extend the previous results to a more general class of Banach spaces, it becomes necessary to prove, at least under certain additional conditions, the existence of an invariant mean on the entire space $B(S,X)$ of bounded functions. That this can always be achieved for the class of Banach spaces which are complemented in their bidual is the following result which is not explicitly stated but can be gleaned from the literature. A Banach space which is complemented in its bidual shall be called an ultrasummand. 
\begin{prop}\label{prop1}
Given a (left) amenable semigroup $S$ and an ultrasummand $X,$ there always exists a (left)-invariant mean $M:B(S,X)\to X.$
\end{prop}
\begin{proof}
Let $J_X:X\to X^{**}$ be the canonical embedding and denote by $\hat{X} =J_X (X)$ the canonical image of $X$ inside $X^{**}.$The given conditions on $S$ and $X$ yield the existence of a continuous linear projection $P:X^{**}\to \hat{X}$  and a left invariant mean $m$ on $S$. Define $\hat{M}:B(S,X)\to X^{**}$ by
\begin{align}\label{star3}
\hat{M}(f)(x^{*})=m(x^{*}\circ f), \quad x^{*}\in X^{*}
\end{align}
Then the map $M:B(S,X)\to X$ given by $M=J_{X}^{-1}\circ P\circ\hat{M}$ is easily shown to be linear and left-invariant.
\end{proof}
\begin{remark}\label{r7}
Unlike in the case of reflexive Banach spaces, it does not necessarily follow that the condition: $M(f)\in\overline{\textnormal{co}}(f(S))$ also holds for invariant means $M$ on ultrasummands. However, for bounded functions taking values in a dual space $X=Z^{*},$ it can still be arranged that $M(f)\in\overline{\textnormal{wco}}(f(S)),$ the closed convex hull in the weak$^{*}$-topology. Indeed, given an invariant mean $m$ on $S,$ we define the invariant mean $M:B(S,X)\to X$ by: $M(f)x=m(\hat{x}\circ f),$ $x\in Z.$ Now assuming to the contrary that $M(f)$ lies outside of $\overline{\textnormal{wco}}(f(S)),$ an application of the Hahn Banach separation theorem applied to $M(f)$ yields that there exists $z\in X$ such that $\hat{z} (x)<\hat{z}(M(f))=M(f)(z)$ for all $x\in  \overline{\text{wco}} (f(S)).$ We can choose $\epsilon>0$ such that 
\begin{align*}
\hat{z}(x)+\epsilon\leq M(f)(z),\quad x\in\overline{\textnormal{wco}}(f(S)).
\end{align*}
In particular,     
\begin{align*}
\hat{z}(f(s))+\epsilon\leq M(f)(z),\quad s\in S.
\end{align*}
Applying $m$ on sides gives
\begin{align*}
M(f)z+\epsilon=m(\hat{z}\circ f)+\epsilon\leq M(f)(z),
\end{align*}
an obvious contradiction.  
\end{remark}
The following theorem due to T. Kania which shows that the converse of Proposition \ref{prop1} also holds provides a complete answer to Question 2.
\begin{theorem}\textnormal{\cite{9}} \label{th10}
For a Banach space $X,$ the following statements are equivalent.
\begin{enumerate}
\item[(i)] For each amenable semigroup $S,$ there exists an $X$-valued invariant mean on $S.$
\item[(ii)] $X$ is an ultrasummand.
\end{enumerate}
\end{theorem}
In particular, the above theorem applies to the classes of Hilbert spaces on the one hand and injective Banach spaces, on the other. Both these classes of Banach spaces are known to satisfy a certain ball intersection property. Below we show that the class of Banach spaces given by the previous theorem also enjoy this ball intersection property.
\begin{prop}
Let $X$ be an ultrasummand. Then each family of closed balls in $X$ such that any finite subfamily of it has a non-empty intersection has a point common to all the balls.
\end{prop}
\begin{proof}
Let $\{B_{r_i}(x_i):i\in\Lambda\}$ be a family of closed balls such that for each finite subfamily $J\subset \Lambda,$ we have
$\cap_{i\in J}B_{r_i}(x_i)\neq \phi.$ To show that $\cap_{i\in\Lambda}B_{r_i}\neq \phi,$ consider the semigroup $S$ defined as follows:
\begin{align*}
S=\{f:\Lambda\to\{0,1\};f^{-1}(1) ~\text{is finite}\}
\end{align*}
where
\begin{align*}
(f+g)(i)=\begin{cases} &0,\quad \text{if}~f(i)=g(i)=0\\&1,~~\quad~\text{otherwise}\end{cases}.
\end{align*}

Let $i_0\in\Lambda$ and for each subset $J\subset \Lambda,$ pick $y_j\in \cap_{i\in J }B_{r_i}(x_i)\cap B_{i_0}(x_{i_0}).$ Define $F:S\to X$ by the formula: $F(f)=y_j$ where $J=\{j\in\Lambda; f(i)=1\}.$ Clearly $F\in B(S,X).$ By Theorem \ref{th10}, we can choose an invariant mean $M:B(S,X)\to X.$ Fix $k\in\Lambda$ and consider the function $h=\chi_{\{k\}}$ in $S.$ It follows that for each $g\in S,$ $(g+h)(k)=1$ and, therefore, $k\in J_{0}=\{j\in\Lambda; (g+h)(j)=1\}.$ Finally, uisng translation invariance of $M,$ we get $M(F)=M(F^{h})\in B_{k}(x_k)$ and this completes the proof.
\end{proof} 
\begin{remark}
\textnormal{ It is not known if the converse of the previous proposition holds. However, the converse holds for Banach spaces not containing copies of $\ell_1.$ In view of Theorem \ref{th10}, therefore, the problem reduces to the following. }
\end{remark}
\noindent Problem 1:  Let $X$ be a Banach space such that each family of closed balls in $ X$ has a non-empty intersection whenever each finite subfamily of it has it. Does it follow that for each amenable semigroup $S,$ there exists an $X$-valued invariant mean on $S$?

We now sketch a proof involving an affirmative answer to Question 1 which holds for a more general class of Banach spaces that include reflexive Banach spaces.
\begin{theorem}
Let $S$ be a (left) amenable group and $X=Z^{*}$ a dual Banach space. Let $W_{wbc} (X)$ be the family of all nonempty weak$^{*}$ closed, bounded convex subsets of $X$ and $F:S\to W_{wbc} (X)$ a given multi-valued map. Then $F$ admits an additive selection if (and only if) there exists a function $f:S\to X$ such that
\begin{align}\label{star4}
f(x+t)-f(t)\in F(s),\quad s,t\in S.
\end{align}
\end{theorem}
\begin{proof}
If $g$ is an additive selection of $F,$ then $g(s+t)-g(t)=g(s)\in F(s)~,s,t\in S$ and the given condition \eqref{star4} is satisfied. Conversely, assume that \eqref{star4} holds for a mapping $f:S\to X.$ By Remark \ref{r7}, there exists a (left 'quasi') invariant mean $M:B(S,X)\to X$ such that $M(g)\in \overline{\text{wco}}(g(S))$ for each $g\in B(S,X).$ For $s\in S,$ define
\begin{align*}
f_{(s)}(t)=f(s+t)-f(t), \quad t\in S.
\end{align*}
By \eqref{star4}, $ f_{(s)} (t)\in F(s)$  for each $t\in S.$ The principle of uniform boundedness applies to yield that f$_{(s)} \in B(S,X).$ 
Define
\begin{align*}
g(s)=M(f_{(s)}),~s\in S.
\end{align*}
It follows that $g(s)\in\overline{\text{wco}}(f_{(s)} (S)) \subset \overline{\text{wco}}(F(s)) =F(s),$ showing thereby that $g$ is indeed a selection of $F.$ Further, for $u,v$ and $t\in S,$ we note that
\begin{align*}
f_{(u+v)}(t)&=f(t+u+v)-f(t)\\ &=f(t+u+v)-f(t+v)+f(t+v)-f(t)\\ &=f_{(u)}^{v}(t)+f_{(v)}(t).
\end{align*}
From this it follows easily that $h$ is additive. Indeed, by the left invariance of $M,$ we have
\begin{align*}
g(u+v)=&M(f_{(u+v)})=M(f_{(u)}^{v}+f_{(v)})+M(f_{(u)}^{v})+M(f_{(v)})\\&=M(f_{(u)})+M(f_{(v)})=g(u)+g(v).
\end{align*}
\end{proof}
\begin{corollary}
Assume that $X$ is a dual space and $f:S\to X $ a mapping satisfying
\begin{align}\label{star5}
\|f(x+y)-f(x)-f(y)\|\leq \epsilon, \quad \text{for all}~ x,y\in S,
\end{align}
for some $\epsilon>0.$ Then there exists an additive mapping $g:S\to X$ such that  $\|f(x)-g(x)\|\leq \epsilon,$ for all $x\in S.$
\end{corollary}
\begin{proof}
Consider the multi-valued map $F:S\to W_{wbc} (X)$ defined by
\begin{align*}
F(x)=B_{\epsilon}(f(x))=\{z\in X:\|z-f(x)\|\leq \epsilon\}.
\end{align*}
$F$ is well defined by virtue of Banach Alaoglu theorem, so each set $F(x)$ being weak$^{*}$- compact, is weak$^{*}$-closed, bounded and also convex. By \eqref{star5}, $f(x+y)-f(y)\in F(x)$ and, therefore, by the previous theorem there exists an additive selection $g$ of $F,$ that is$ \|f(x)-g(x)\|\leq \epsilon,$ for all $x\in S.$
\end{proof}
\begin{remark}
\textnormal{
In view of the above theorem, it is tempting to surmise that Hyers-Ulam stability might also hold for the Cauchy equation defined on an amenable semigroup and taking values in an ultrasummand. Unfortunately, the above method as applied in the case of a dual Banach space cannot be used for functions taking values in an ultrasummand. This is because in the case of dual Banach spaces $X,$ the condition: $M(f)\in\overline{\text{wco}}  (f(S))$ which follows from the definition of an invariant mean $M$ on $B_{wbc} (S,X)$ using an invariant mean $m$ on $S$ satisfies the equation: $x^{*} (M(f))=m(x^{*}\circ f),~x^{*} \in X^{*}.$ However, it can be shown that the analogous equation: $\hat{M} (f)(x^{*})=m(x^{*} \circ f),~ x^{*}\in X^{*}$ which is used to define an invariant mean on an ultrasummand $X$ yields the weaker assertion that $M(f)\in \overline{\text{wco}}  (f(S)),$  where the weak$^{*}$ closure is taken in the bidual $X^{**}$ of $X$ and that there is no guarantee that the set is mapped onto a weakly closed subset of $X$ by the projection $P:X^{**}\to X.$ This suggests the following problem which seems to be open.  }
\end{remark}
\noindent Problem 2: Does the stability of the Cauchy equation hold in an ultrasummand?.  
 
Let us record below the main ingredient in the proof of the above theorem of Badora on the stability of Cauchy's equation in injective Banach spaces which depends upon the following refinement of Ger's theorem on the existence of an invariant mean as a selection of a multi-valued map.
\begin{theorem}\textnormal{\cite{2}}
Let $X$ be an injective Banach space and $S$ a (left) amenable semigroup. Further, let $F:B(S,X)\to C$ be a map where $C$ is the family of all subsets of $X$ which are intersections of closed balls containing $f(S), f\in B(S,X).$ Then there exits an invariant mean $M:B(S,X)\to X$ such that $M(f)\in F(f),~f\in B(S,X).$
\end{theorem}
\noindent e. \textit{Banach limits and the structure of Banach spaces}

The following example provides a simple application of how Banach limits are typically used in Banach space theory to extract useful information about their structure.

Notation: Given a bounded function $f\in B(S,X)$ and an invariant mean $M$ on $B(S,X),$ in what follows it will be more convenient to use the integral symbol for the action of $M$ on $f:$ thus $M(f)$ shall be denoted by $\int_{X} fdx$ or by $\int_{X} f(x)dx$ to indicate the dummy variable.  
\begin{example}
\textnormal{
Let $X$ be a Banch space and denote by $\textnormal{Lip}(X)$ the space of all \textit{Lipschitz} functions $f:X\to \mathbb{R},$ vanishing at a distinguished point, say $\theta\in X.$ It is easily checked that under pointwise operations, $\textnormal{Lip}(X)$ is a Banach space when equipped with the norm:
\begin{align*}
\|f\|=\underset{x\neq y}{\sup}\dfrac{|f(x)-f(y)\|}{d(x,y)}.
\end{align*}
Further, we note that $\text{Lip}(X)$ contains $X^{*}$ isometrically as a subspace and more importantly, as a complemented subspace. Indeed, let $M$ be an invariant mean on $X.$ Using the shift invariance of $M,$ it is easily checked that the map $P:\text{Lip}(X)\to X^{*}$ given by 
\begin{align*}
P(f)(x)=\int_{X} f(x+y)-f(y) dy,~~f\in\text{Lip}(X),~x\in X
\end{align*}
defines a linear continuous projection.}
\end{example}
A highly nontrivial application of Banach limits pertains to the existence of Lipschitz selections involving certain set-valued maps defined on a metric space and taking values in the collection $H(X)$ consisting of nonempty, closed, bounded and convex subsets of a Banach space $X$ equipped with the Hausdorff metric:
\begin{align*}
d_{H}(A,B)=\sup(\{d(x,A);x\in B\}\cup\{d(x,B);x\in A\}).
\end{align*}
The well-known Michael's selection theorem ([see \cite{3}, Chapter 1])says that given a Banach space $X,$ it is always possible to choose a selector $\psi:H(X)\to X,$ (i.e. $\psi(A)\in A$ for each $A\in H(X)$) which is always continuous. Equivalently, for each metric space $M$ and a multivalued map  $F:M\to H(X),$ there exists a 'selection' map $f:M\to X$ of $F$ in the sense that $f(x)\in F(x)$ for each $x\in M.$

This motives the next question which is natural:

\noindent \textbf{Question}: Given a Banach space $X,$ does there exist a Lipschitz (uniformly continuous) selector? \\
Answer: Yes, if $X$ is finite dimensional.\\
\noindent \textbf{Conjecture}: The above statement also holds for all infinite dimensional Banach space $X$!

However, the question was settled in the negative by Positselskii \cite{11A} in 1971 and independently (for the uniformly continuous case) by Przeslawskii and Yost \cite{10} in 1989.

In the case of finite dimensional spaces, the desired Lipschitz selector is given by the Steiner point map $s:H(\ell_{2}^{n})\to \ell_{2}^{n}$ which is defined by (the vector valued integral):
\begin{align*}
s(A)=n\int_{X} h_A(x)xd\sigma (x),
\end{align*}
where the integral is computed over the unit sphere $S^{n-1}$ in $\ell_{2}^{n}$ with respect to $\sigma,$ the normalised Lebesgue measure on $S^{n-1}$ and $h_A$ is the support functional of $A$ given by
\begin{align*}
h_A(x)=\sup\{\langle x,a\rangle:a\in A\}, ~~x\in S^{n-1}.
\end{align*}
It can be checked that $s(A)\in\ell_{2}^{n},$ so $s$ defines a selector.  Further $s$ is Lipschitz with its Lipschitz constant being given by: $2\pi^{-1/2}\Gamma(n/2+1)/\Gamma((n+1)/2)).$ (For a proof of these assertions, see \cite[Chapter 2]{3}. See also \cite{11A}).
\begin{theorem}\label{th13}
Given an infinite dimensional Banach space $X,$ there is no selector $H(X)\to X$ which is Lipschitz.
\end{theorem}
\begin{proof}(Sketch): The proof depends upon the following ingredients:                                                    
\begin{enumerate}
\item[(a)]  $H(X)$ is an absolute Lipschitz retract (ALR). More precisely, $H(x)$ is an 8-(ALR).
\item[(b)] 	Dvoretzky's spherical sections theorem: Dvoretzky's spherical sections theorem: Each infinite dimensional Banach space $X$ contains for each $n,$ an $n$-dimensional subspace $X_n$ which is 2-isometric with $\ell_{2}^{n}.$
\end{enumerate}
The statement in (a) was proved by Przeslawskii and Yost \cite{11} whereas (b) is a well-known result in the local theory of Banach spaces which was proved by Dvoretzky in 1960. (See, e.g.  \cite{3}, Theorem 12.10 for a proof).

We break the proof into several steps.

Step 1: Here we show that the theorem holds for the special case where $X$ is an infinite dimensional Hilbert space. In fact, it follows from (a) that there does not even exist a Lipschitz retraction $H(X)\to X.$ This follows because (a) yields that $X$ is an (ALR) (=1-injective) if (and only if) there exists a Lipschitz retraction $r:H(X)\to X.$  But it is well known that a Hilbert space can never be injective, unless it is finite dimensional. One of the ways to see this is to note that $\ell_\infty$ contains $\ell_2$ isometrically and that each bounded linear operator from $\ell_\infty$ into $\ell_2$ is 2-summing. Thus assuming $\ell_2$ to be injective leads to the identity map on $\ell_2$ being 2-summing which is impossible. Alternatively, the presumed injectivity of $X$ combined with the Banach Alaoglu theorem yields $X$ as a complemented subspace of $C(K),$ with $K$ being the (weak$^{*}$-compact) unit ball of $X^{*}.$ Recalling the well-known fact that the Banach space $C(K)$ has the Dunford-Pettis property and that this property is inherited by complemented subspaces, it follows that (the Hilbert space) $X$ has it, which is impossible.

Step 2: Given a sequence of Lipschitz retracts $r_n:H(\ell_{2}^{n})\to \ell_{2}^{n},$ then $L(r_n)\to \infty.$\\
\textit{Proof}: Suppose, to the contrary, there exists $c>0$ such that $L(r_n)\leq c$ for all $n\geq 1,$i.e., $\|r_n (A)-r_n (B)\|\leq cd_H (A,B)$ for all $A,B\in H(\ell_{2}^{n})$ and $n\geq 1.$ Let $P_n: \ell_{2}\to\ell_{2}^{n}$ be the natural projection where we identify $x= (x_i)_{i=1}^{\infty}\in \ell_2$ with $x^{(n)}=(x_i)_{i=1}^{n}\in \ell_{2}^{n}:P_{n}(x)=x^{(n)}.$ Choose an invariant mean $M:B(\ell_2)\to \ell_2$ : $M(f)=\int_{X} f dn\in\overline{ \text{co}}(f(\mathbb{N}))$ for $f\in B(\ell_2).$ Such an invariant mean exists by Theorem \ref{thm9}. Define $r:H(\ell_2 )\to \ell_2$ by 
\begin{align*}
r(A)=\int_{X} g_A(n)dn
\end{align*}
where $g_A (n)=r_n (P_n (A)).$ Note that for $x\in \ell_2,$ since $x^{(n)}\to x,$ we get
\begin{align*}
r(x)=\int_{X} g_x dn\in\overline{\text{co}}(g_x(\mathbb{N}))=\overline{\text{co}}(x^{(n)}; n\geq 1).
\end{align*}
\end{proof}
By shift invariance of $M,$ we note that $r(x)\in\overline{\text{co}} (x^{(n)};n\geq k),$ for all $k\geq 1.$ This gives
\begin{align*}
r(x)\in\cap_{k=1}^{\infty}\overline{\text{co}}(x^{(n)}; n\geq k)=x.
\end{align*}
Thus $r$ is a retraction. Finally, given $A,B\in H(\ell_2 )$ and using that $\|\int_{X} dn\|=1,$ we get 
\begin{align*}
\|r(A)-f(B)\|&=\left|\int_{X} r_n(P_n(A))-r_n(P_n(B)) dn\right\|\\&\leq \|r_n(P_n(A))-r_n(P_n(B))\|\\&\leq cd_H(P_n(A),P_n(B))\leq cd_H(A,B).
\end{align*}
In other words, $r$ is a Lipschitz retraction, contradicting the assertion proved in Step 1.

Step 3: Suppose to the contrary that there exists a Lipschitz retraction $r:H(X)\to X.$ Choose $X_n$ according to (b). Thus there exists a linear isomorphism $\phi_n:\ell_{2}^{n}\to X_n$ such that $\|\phi_n\| \|\phi_{n}^{-1}\|<2.$ Denote the restriction of $r$ on $X_n$ by $s_n$ and define $\Psi:H(\ell_{2}^{n})\to \ell_{2}^{n}$ by $\Psi(A)=\phi_{n}^{-1} (A).$ Then 
$r_n:H(\ell_{2}^{n})\to \ell_{2}^{n}$ defined by: $r_n (A)=\phi_n \circ s_n \circ \Psi(A)$ is seen to be a Lipschitz retraction such that 
\begin{align*}
\|r_n(A)-r_n(B)\|\leq 2L(r)d_H(A,B)
\end{align*}
 and so, $L(r_n)\leq 2L(r).$ But this contradicts the assertion in Step 2 and the proof is completed.

\begin{remark}
\textnormal{
With some extra effort, it can also be deduced that for an infinite dimensional Banach space $X,$ there does not exist even a uniform selector: $H(X)\to X.$  Also, there does not even exist a Lipschitz selector $K(X)\to X,$ where $K(X)$ denotes the sub-collection of $H(X)$ consisting of compact sets. This motivates the following question.
}
\end{remark}
\noindent \textbf{Question}:  Under certain conditions on $X,$ is it possible to choose uniform retractions $K(X)\to X.$
\begin{theorem}
Given that $X$ is a uniform retract of $K(X),$ then $X^{**}$ is injective.
\end{theorem}
\noindent Problem 3: Does the converse of the previous statement hold?
\begin{remark}
\textnormal{ It follows from Theorem \ref{th13} and the discussion preceding it that the property involving the existence of a Lipschitz selector $H(X)\to X$ is a finite dimensional (FD) property in the class of Banach spaces: a property (P) is said to be a (FD) property if it holds for finite dimensional Banach spaces but fails in each infinite dimensional Banch space. A comprehensive study of this phenomenon was undertaken in \cite{15} where it was shown with the help of a host of examples that suitably formulated analogues of (FD) properties in the framework of Fr\'{e}chet spaces lead to new insights into the structure of the latter class of spaces. It is interesting to note that in a fairly large number of cases, these (FD) properties hold in this new setting exactly if the underlying Fr\'{e}chet space is nuclear! In the light of this discussion, the following question arises naturally.} \end{remark}

\noindent Problem 4: Given a nuclear Fr\'{e}chet space $X,$ does it follow that there exists a Lipschitz selector $H(X)\to X$? How about the converse?

Regarding extension of nonlinear maps, we have
\begin{theorem}\textnormal{(McShane, see \cite{16} for a proof)}
 Every Lipschitz function on a subset of a metric space can be extended to a Lipschitz function on the whole space.
\end{theorem}
\begin{remark}
\textnormal{ 
Sometimes it becomes necessary to know if in the above theorem, the choice of the extension Lipschitz map could be effected in a continuous linear fashion; in other words, whether the 'extension' map $F: \text{Lip}(Z)\to \text{Lip}(X),$ i.e.,  $F(g) |_{Y}=g$ for each $g\in \text{Lip}(Z)$ can be chosen to be bounded and linear. This is obviously seen to be the case in Hilbert spaces. However, the question whether the converse of this statement also holds has recently been settled by the author in \cite{17}, whereas the linear case involving Hahn Banach extensions of linear functionals was proved by Fakhoury \cite{5} in 1972.
}
\end{remark}
\begin{theorem}\label{th18}\textnormal{(\cite{5}, see also \cite{6})}: 
Let $X$ be a Hilbert space. Then for each subspace $Y$ of $X,$ there exists a bounded linear 'extension' map $F: Y^{*}\to X^{*},$ i.e, $F(g) |_{Y}=g,$  $g\in Y^{*}.$ Conversely, given the existence of such a bounded linear extension map corresponding to each subspace $Y$ of $X,$ it holds that $X$ is a Hilbert space. 
\end{theorem}
In what follows, we provide a brief sketch of the proof of the Lipschitz analogue of the previous theorems involving the existence of a bounded linear extension operator $F:\text{ Lip}(Z)\to \text{Lip}(X)$ acting between spaces of Lipschitz functions. More precisely, we prove the following: 
\begin{theorem}\label{th19}\textnormal{(\cite{17})}
Let $X$ be a Banach space and $Z$ a subspace of $X$ such that there exists a bounded linear 'extension' map $F: \textnormal{Lip}(Z)\to \textnormal{Lip}(X),$ i.e.,  $F(g) |_{Y}=g$ for each $g\in  \textnormal{Lip}(Z).$ Then $X$ is a Hilbert space. 
\end{theorem}
Using Banach limits, we can 'linearise' the above property and reduce it to the existence of a bounded linear extension map $G: Z^{*}\to X^{*}.$ 

\noindent (\textit{Sketch of proof:}) We begin by using $F$ to define another Lipschitz map $G: \text{Lip}(Z)\to \text{Lip}(X)$ which actually takes values in $X^{*}$ and when restricted to $Z^{*}$ yields an extension map $G: Z^{*}\to X^{*}$ which is Lipschitz. Using the existence of a Banach limit $\int_{X}\dot dx$ on $\ell_\infty (X),$ the new map $G: \text{Lip}(Z)\to \text{Lip}(X)$ is defined as follows:
 \begin{align*}
 G(f)(z)=\int_{X}\left\{\int_{X}\left[(Ff)(x+y+z)-(Ff)(x+y)\right]dy\right\}dx,~~\quad f\in\textnormal{Lip}(Z),~z\in X.
 \end{align*}
The Lipschitz property of $f$ yields that the above expression on the (RHS) of the equation defines a Lipschitz map on $X:$ for $z_1, z_2\in X,$ we have
\begin{align*}
\|G(f)z_1-G(f)z_2\|\leq \text{Lip}(F)~\text{Lip}(f)\|z_1-z_2\|.
\end{align*}
Further, since
\begin{align*}
G(f)(z_1+z_2)&=\int_{X}\left\{\int_{X} [(Ff)(x+y+z_1+z_2)-(Ff)(x+y)]dy\right\}dx\\& =\int_{X}\left\{\int_{X} [(Ff)(x+y+z_1+z_2 )-(Ff)(x+y+z_2 )]dy\right\}dx\\&\qquad\qquad+\int_{X} \left\{\int_{X} [(Ff)(x+y+z_2 )-(Ff)(x+y)]dy\right\}dx\\ &=\int_{X} \left\{\int_{X} [(Ff)(x+y+z_1 )-(Ff)(x+y)]dy\right\}dx\\&\qquad\qquad +\int_{X}\left\{\int_{X} [(Ff)(x+y+z_2 )-(Ff)(x+y)]dy\right\}dx\\& =G(f)(z_1 )+G(f)(z_2 ),
\end{align*}
it follows that $G$ takes its values in $X^{*}.$ Thus, restricting to $Z^{*}$ gives a bounded linear map, again denoted by $G: Z^{*}\to X^{*}$ which can be shown to be an 'extension map': $G(g)(z)=g(z),~ g∈\in Z^{*}$ and $z\in Z.$ By Theorem \ref{th18}, $X$ is a Hilbert space.

Combining the previous theorem with the linear version of the Hahn Banach selection theorem yields the following corollary:
\begin{corollary}\textnormal{(\cite{17})}
A Banach space has the property that each closed and convex subset of $X$ is a Lipschitz retract of $X$ if and only if $X$ is isomorphic to a Hilbert space.
\end{corollary}
\begin{corollary}\textnormal{(\cite{17})}
Let $X$ be a Banach space and $Y$ a subspace of $X$ such that there exists a bounded linear map 'extension' map $L: \textnormal{Lip}(Y)  \to  \textnormal{ Lip}(X),$ i.e.,  $L(g) |_{Y}=g$ for each $g\in \textnormal{Lip}(Y). $ Then $Y$ is 'locally complemented'.
\end{corollary}
A subspace $Y$ of a Banach space $X$ is said to be locally complemented if there exists $c>0$ such that for each finite dimensional subspace $M$ of $X,$ there exists a continuous linear map $f:M\to Y$ with $\|f\|\leq c$ and $f(x)=x$ for all $x\in M\cap Y.$

\begin{remark}
\textnormal{ It is known \cite{5} that $Y$ is locally complemented in $X$ if and only if for each (closed) subspace $Y$ of $X,$ there exists a bounded linear 'extension' map $\psi: Y^{*}\to X^{*},$ $\psi(g) |_{Y}=g$ for each $g\in  Y^{*}.$}
\end{remark}
The previous corollary motivates the question whether local complementedness of a subspace Y of a Banach space X implies that it is a Lipschitz retract.

The previous corollary motivates the question whether local complementedness of a subspace $Y$ of a Banach space $X$ implies that it is a Lipschitz retract. It can be shown that this question is equivalent to one of the most important open problems in nonlinear Banach space theory:

\noindent \textbf{Question}: Given a Banach space $X,$ is it true that $X$ is a Lipschitz retract of $X^{**}.$

In 2009, Kalton showed that the answer is in the negative:\\
Theorem (\cite{8}): There exists an non-separable Banach space $X$ which is not a Lipschitz retract of $X^{**}.$ However, the question remains open in the separable case!

\section*{Acknowledgement}
Part of this work was carried out at the Harish Chandra Research Institute, Allahabad while the author was visiting the institute during Dec. 2018-Feb. 2019. He wishes to record his thanks to his host Prof. Ratnakumar for his kind invitation and hospitality during the authors stay at the institute. Thanks are also due to the referee for careful reading of the manuscript and for his suggestions. 



\end{document}